%% file: pmlr-sample.tex
\documentclass[eat,twocolumn]{jmlr}

\usepackage{longtable}

\usepackage{booktabs}
\usepackage{csquotes}
\usepackage{amssymb,amsmath}
\usepackage{algorithm}
\usepackage{algpseudocode}

\usepackage[load-configurations=version-1]{siunitx} 

\theorembodyfont{\upshape}
\theoremheaderfont{\scshape}
\theorempostheader{:}
\theoremsep{\newline}

\newtheorem{assume}{Assumption}

\jmlrvolume{1}
\firstpageno{1}

\jmlryear{2022}
\jmlrworkshop{Machine Learning for Health (ML4H) 2022}

\title[Personalized Dose Guidance using Safe Bayesian Optimization]{Personalized Dose Guidance  using Safe Bayesian Optimization}

 \author{\Name{Dinesh Krishnamoorthy} \Email{d.krishnamoorthy@tue.nl}\\
  \addr  Department of Mechanical Engineering, Eindhoven University of Technology, Eindhoven 5600 MB,  The Netherlands
  \AND
  \Name{Francis J. {Doyle III}} \Email{frank\_doyle@seas.harvard.edu}\\
  \addr Harvard John A. Paulson School of Engineering and Applied Sciences, Allston , 02134 MA, USA
 }

\begin{document}

\maketitle

\begin{abstract}
	This work considers the problem of personalized dose guidance using Bayesian optimization that learns the optimum drug dose tailored to each individual, thus improving therapeutic outcomes. Safe learning using interior point method ensures patient safety with high probability. This is demonstrated using the problem of learning the optimum bolus insulin dose  in patients with type 1 diabetes to counteract the effect of meal consumption. Starting from no \textit{a priori} information about the patients, our dose guidance algorithm is able to improve the therapeutic outcome (measured in terms of \% time-in-range) without jeopardizing patient safety. Other potential healthcare applications are also discussed.
\end{abstract}

\section{Introduction}
Treatment of a disease or medical condition often involves administering therapeutic agents (e.g. drugs, radiation) to trigger a desired response. Often precise doses of therapeutic agents  must be administered in order to get better treatment outcomes. However, the triggered response  of a certain therapeutic agent can vary significantly from one individual to the other due to several  physiologic, genetic, and/or environmental factors.
This affects the efficacy of a drug from one individual to another. Moreover, ensuring patient safety  (e.g. overdose) further adds to this challenge. 

As such, for a particular dose, some patients may have very little therapeutic benefit, while others may have unwanted side effects due to drug overdose. Knowing the right dose of therapeutic agents to trigger a desired response in a given individual is not trivial. Pharmacokinetic and pharmacodynamic models developed based on a population level, or  clinical data averaged over a cohort of patients are often  not well suited to find the optimal drug dose for a particular individual \citep{tucker2017personalized}. Targeted treatment that accounts for the individual variations is key to improve therapeutic outcomes. 
Personalized medicine is an emerging practice in medicine that uses information from the individual to guide therapeutic decisions. Personalizing the dose for each individual, ensures that the patients receive the right dose of the therapeutic agent, thus improving the treatment outcomes, mitigate unwanted side effects, and reduce  healthcare costs.

For example, patients with type 1 diabetes need to administer  rapid-acting insulin  to counteract the effect of the meal consumption (known as bolus insulin). Even for the same meal, the insulin needs can vary significantly from one individual to the other. Effective management of type 1 diabetes  thus requires finding the right dose of bolus insulin personalized to the individual in order  to counteract the effect of meal consumption, and maintain the blood glucose concentration within a desired target range.  This work  shows that starting from no prior knowledge about the patient, nor historical patient-data, safe Bayesian optimization \citep{DK2022constrainedBayesOpt} is able to learn the optimal  insulin dose tailored for each individual, without jeopardizing patient safety, thus improving the therapeutic outcome, measured as percentage of  time spent within the target range for the glucose concentration \citep{battelino2019clinical}.

\section{Technical details}

To develop personalized dose guidance algorithm, firstly we need some performance metric that quantifies the quality of the  response triggered by the dose (i.e., the reward).  In addition to the dose $ x \in \mathcal{X} $, the performance metric may also depend on external factors or contexts $ d  \in \mathbb{R}^{n_{d}}$. Let the reward function  be denoted by $ f_{0}(x,d)  $. Secondly, suitable metrics that quantify safety-critical constraints, denoted by $ f_{i}(x,d) $ for all $ i=1,\dots,m $, are also required in order to account for patient safety. The optimal dose guidance problem can then be formulated as
\begin{align}\label{Eq:Opt}
	x^* = \arg \max_{x\in \mathcal{X}} \; &f_{0}(x,d) \\
	\textup{s.t.} \; & f_{i}(x,d) \ge0 ,\; \forall i=1,\dots,m	\nonumber
\end{align}
However,  obtaining accurate models  $ \{f_{i}(x,d)\}_{i=0}^m $ tailored to an individual patient is prohibitively time consuming and expensive for the patient, healthcare personnel, and the engineers.  
If the performance metrics and the constraints can be observed after administering  a dose on an individual, then we can learn the optimum dosing needs for that individual by sequentially evaluating different doses (by carefully trading-off exploration and exploitation), and find the dose that gives the best response.  

Bayesian optimization is one such black-box decision-making framework that learns the optimum of the unknown, but observable system, by sequentially querying the system \citep{shahriari2015taking,jones1998efficient}.  
Bayesian optimization uses probabilistic surrogate models, typically Gaussian processes (GP), to capture our belief of the unknown functions.The GP posteriors are obtained by conditioning on the observations. 
The GP posterior conditioned on the data set $ \mathcal{D}^{k-1} $ is then used to induce an acquisition function $ \mathcal{A}: \mathcal{X} \times \mathbb{R}^{n_{d}} \rightarrow \mathbb{R}  $, that is indicative of how desirable evaluating a dose $ x $ is expected to be with respect to the reward function $ f_{0}(x,d) $ \citep{shahriari2015taking}. Learning from patient-specific data enables personalized decision-making. 

In the presence of safety-critical constraints that cannot be violated, one needs to incorporate safe learning when computing the  next dose $ x^{k+1} $ that needs to be evaluated. 
To account for the constraints, Gaussian process surrogates  are placed for each safety-critical constraint. Let the mean and the std. deviation for the $ i^{th} $ constraint
GP conditioned on the past $ k-1 $   observations be denoted by $ \mu_i^{k-1} $ and $ \sigma_i^{k-1} $, respectively. After  $ k-1 $ evaluations, the safe region is partially revealed, which can be quantified as \[ \hat{\mathcal{F}}^{k-1} := \{x \in \mathcal{X} | \mu_i^{k-1}(x,d)  - \sigma_i^{k-1}(x,d) \geq 0 \} \] 
such that $ \hat{\mathcal{F}}^{k-1} \subseteq \mathcal{X}$. 

The  key idea of incorporating safe learning using interior-point method is then as follows.
Appending any suitable acquisition function $ \mathcal{A}(x,d) $ with log barrier terms $  \mathcal{B}_{i}(x,d) := \ln[\mu_i^{k-1}(x,d)  - \sigma_i^{k-1}(x,d)]   $
ensures that the next query point $ x^k $ remains in the  interior of the partially revealed safe region $ \hat{\mathcal{F}}^{k-1}  $.  Therefore, for the given context $ d^k $, the dose to be evaluated is
\begin{equation}\label{Eq:acq2}
	x^{k}=  \arg \max_{x \in \mathcal{X}} \; \mathcal{A}(x,d^k) - \tau \sum_{i=1}^{m}\mathcal{B}_{i}(x,d^k)
\end{equation}
	for some $ \tau >0 $. 
	We assume that the feasible region $ \mathcal{F} := \{x\in \mathcal{X} | f_{i}(x)\ge 0, \forall i = 1,\dots,m\} \neq \O $, and a safe, but suboptimal action $ x_{0} \in \mathcal{F} $ is known $ \forall d $.
This is a reasonable assumption  since a safe but suboptimal dose would correspond to a drug dose of zero, i.e., no medical intervention.
As such, if the partially revealed safe region is an empty set $ \hat{\mathcal{F}}^{k-1}(d)  = \O$, then the safe but suboptimal dose $ x^{k} = x_{0} $ is applied.
\begin{assume}
	The Gaussian process models used as surrogate for the constraints satisfy \[ |f_{i}(x,d) - \mu_{i}^{k-1}(x,d)| \le \sqrt{\beta^{k}} \sigma_{i}^{k-1}(x,d) \]w.p. at least $ 1-\delta $, $ \forall k>0 $, $ \forall i=1,\dots,m $.
\end{assume}
This assumption can be satisfied by choosing the confidence scaling parameter $ \beta $ as shown in \cite[Theorem 2]{chowdhury2017kernelized}.
\begin{theorem}
	Under Assumption 1,  the dose computed by \eqref{Eq:acq2} satisfies $ f_{i}(x^k,d^k) \ge0$ w.p. at least $ 1-\delta $, $ \forall k>0 $, $ \forall i=1,\dots,m $.
\end{theorem}
\begin{proof}
	Assumption~1  implies that $\forall i \in \mathbb{I}_{1:m} $, $ \forall x \in \mathcal{X} $
\begin{equation}\label{Eq:prf1}
	f_{i}(x,d) \ge \mu^{k-1}_{i}(x,d)  - \sqrt{\beta^{k}_{i}} \sigma^{k-1}_{i}(x,d)  
\end{equation}
holds w.p. at least $ 1-\delta $ for any $ \delta\in (0,1) $. 

If $ \forall k>0  $, $ \exists\;  x^{k}\in  \mathcal{X} $  given by  (2), the log barrier term $ \mathcal{B}_{i}(x,d^k) $  ensures 
\begin{equation}\label{Eq:prf2}
	\mu^{k-1}_{i}(x^{k},d^k)  - \sqrt{\beta^{k}_{i}} \sigma^{k-1}_{i}(x^{k},d^k)  > 0, \quad \forall i\in \mathbb{I}_{1:m}
\end{equation}
holds. Combining \eqref{Eq:prf1} and \eqref{Eq:prf2}  proves our result. 
\end{proof}

\section{\textit{In silico} experimental results}
Type 1 diabetes (T1D) is a metabolic disorder characterized by absolute insulin deficiency, which requires injecting bolus (fast-acting) insulin  along with each meal in order to counteract the spike in the blood glucose concentration induced by the meal (postprandial effect). Bolus dose guidance algorithms help patients compute the optimum insulin dose to match the meal carbohydrates (CHO). The current standard-of-care bolus calculators are based on patient specific
parameters such as insulin-carb-ratio, correction factors, insulin sensitivity etc. \citep{schmidt2014bolus}, which are typically unknown. High uncertainty in such parameters are also the norm rather than the exception \citep{garcia2021advanced}, making it challenging to use current bolus calculators.

This work presents a purely data-driven safe and personalized  insulin dose guidance algorithm based on announced meals, that  does not depend on any patient-specific parameters, or  historical clinical data. Safe learning guarantees patient safety with high probability (cf. Theorem~1).
Continous glucose monitor (CGM) measures the blood glucose concentration, from which the reward and the constraint functions are observed. Since the objective is to reduce the postprandial peak glucose following meal consumption, the reward function is given by the maximum observed glucose concentration after meal consumption, as shown in Fig.~\ref{Fig:sim}a. Hypoglycemia, characterized by glucose concentration below 70mg/dl, is a safety critical constraint, which is also observed from the CGM measurement (cf. Fig.~\ref{Fig:sim}a). The objective is to compute the optimal insulin dose $ x $ for a given meal size (categorized for example as S,M,L,XL).

Our algorithm (summarized in Appendix A) is demonstrated on the 10-adult cohort of the US-FDA-accepted virtual patient simulator \citep{man2014uva}, over a period of three weeks with three meals a day, with randomly varying meal sizes. The \textit{in silico} experimental results  show that, starting from no prior knowledge,  our algorithm is able to learn the optimum dose for each patient, thus maintaining the blood glucose concentration within safe targets, and improve the \% time-in-range (cf. Fig.~\ref{Fig:sim}b), which is the main measure of the therapeutic outcome \citep{battelino2019clinical}. We also plot the 5-95\% quantile to verify safety.

 Regarding patient safety, firstly note that blood glucose concentration between $54<BG<70$mg/dl is considered mild to moderate hypoglycaemia, and a blood glucose level below 54mg/dl is characterized as severe hypoglycaemia. In total there were 5 mild hypoglycemia incidents among the entire cohort, where the glucose went just below 70 mg/dl, but never below 54 mg/dl (i.e. no severe hypoglycemia). To be more precise: Subject 2 had 3 days where the glucose went just below 70mg/dl, and Subject 10 had 2 days where the glucose went just below 70mg/dl. The other subjects had no (mild or severe) hypo incidents.
In order to give a clearer picture regarding the safety, the individual patient data is shown  in  Appendix~B (where these 5 mild incidents can be seen).

We use the  same algorithm with identical hyperparameters across all the subjects, i.e., our algorithm does not require any patient-specific tuning/setup.
 The same algorithm was also successfully tested on a cohort of 50 virtual patients based on the Hovorka T1D model \citep{wilinska2010simulation}. Comparing the performance of our algorithm on two completely
different physiological simulators  demonstrate the robustness
of our algorithm towards patient variability.
 Intra-patient variations  such as diurnal patterns, often observed in patients with T1D \citep{hinshaw2013diurnal} were also  successfully tested by incorporating {mealtime} as additional context, and will be presented and benchmarked against current standard-of-care. 
  Results comparing different constrained Bayesian optimization algorithms will also be presented that shows the benefit of using our proposed interior-point methods to handle the safety-critical constraints. 
 \begin{figure*}
 	\centering
 	\includegraphics[width=0.999\linewidth]{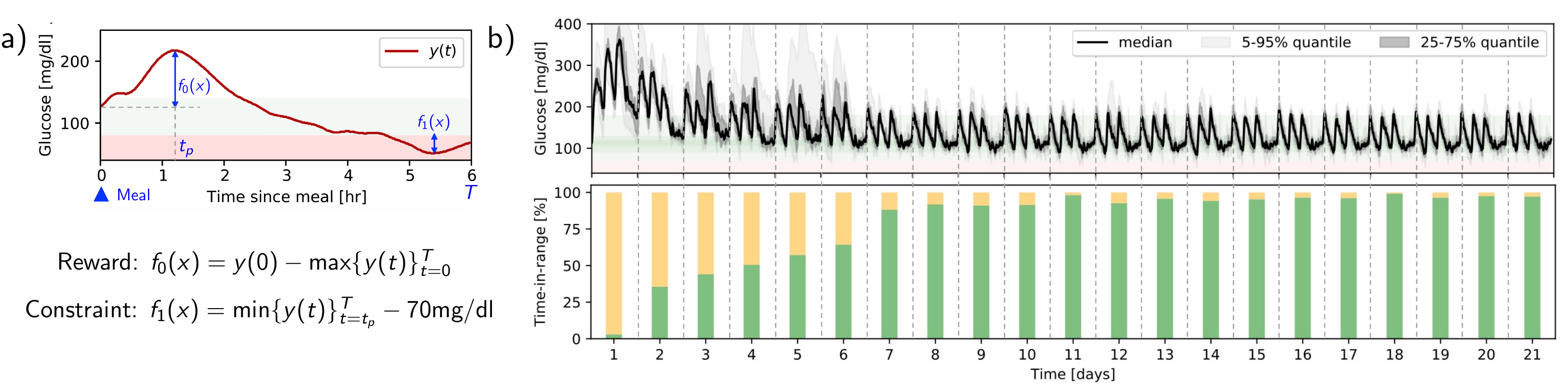}
 	\caption{\small{In silico experimental results on the US-FDA accepted virtual patient cohort. a) Typical postprandial CGM signal  showing the reward $ f_{0} $ and the constraint $ f_{1} $. b) The CGM data and  \% time-in-range (green) and \% time-above-range (yellow)  for the 10-adult cohort  over 3 weeks shows that out algorithm is able to safely learn the optimum insulin dose and improve the treatment outcome. (cf. Appendix B for individual patient data).}  } \label{Fig:sim}
 \end{figure*}
 
\section{Discussion and Conclusion}

This work shows how safe Bayesian optimization using interior point methods can be used to develop personalized insulin dose guidance algorithm to account for the meal consumption in patients with type 1 diabetes. We showed that starting from no prior knowledge about the patient,  Bayesian optimization framework allows to safely learn the dose online by directly interacting with the environment. Online learning  also  allows to easily incorporate human feedback (e.g. pain or discomfort ratings) in the decision-making. For example, in addition to characterizing hypoglycemia constraint as blood glucose below 70mg/dl, it can also be supplemented with the level of symptoms actually experienced by the patient.  

Beyond the application to type 1 diabetes considered in this work, other examples where personalized dose guidance  can be useful include the following. Patients with type 2 diabetes requiring insulin need to  find the optimal basal  (long acting) insulin to keep the fasting blood glucose levels within a target zone \citep{DK2021ESCT2D}. Safe Bayesian optimization as used in this work can be used to titrate the right dose of basal insulin, without causing hypoglycemia. 
Similarly, patients requiring blood thinners to manage thromboembolic disorders need to find the right dose of  anti-coagulants such as warfarin or heparin. Overdosing can  lead to abnormal bruising of excessive bleeding \citep{kuruvilla2001review}, and personalized dose guidance algorithm can be beneficial here.
Cancer treatment using radiotherapy requires finding the right dose of plasma or thermal radiation to kill tumor cells without damaging the healthy cells \citep{luo2015set,bonzanini2020toward}. Safe Bayesian optimization can also be used to find the optimal radiation dose  to reduce side effects, and improve therapeutic outcomes.  As such, personalized dose guidance algorithms using safe Bayesian optimization as presented  in this work can inspire several healthcare applications beyond type 1 diabetes. 

\bibliography{L4DC,diabetes}

\input{appendix}

\end{document}

%% file: appendix.tex
\section*{Appendix A- Bolus calculator algorithm}
Given the CGM data $ y(t) $, and a meal announcement $ m $ at time $ t_{m} $, the bolus calculator  computes the bolus dose $ x $ using the algorithm as summarized below.
\begin{figure}[h]
	\centering
	\includegraphics[width=\linewidth]{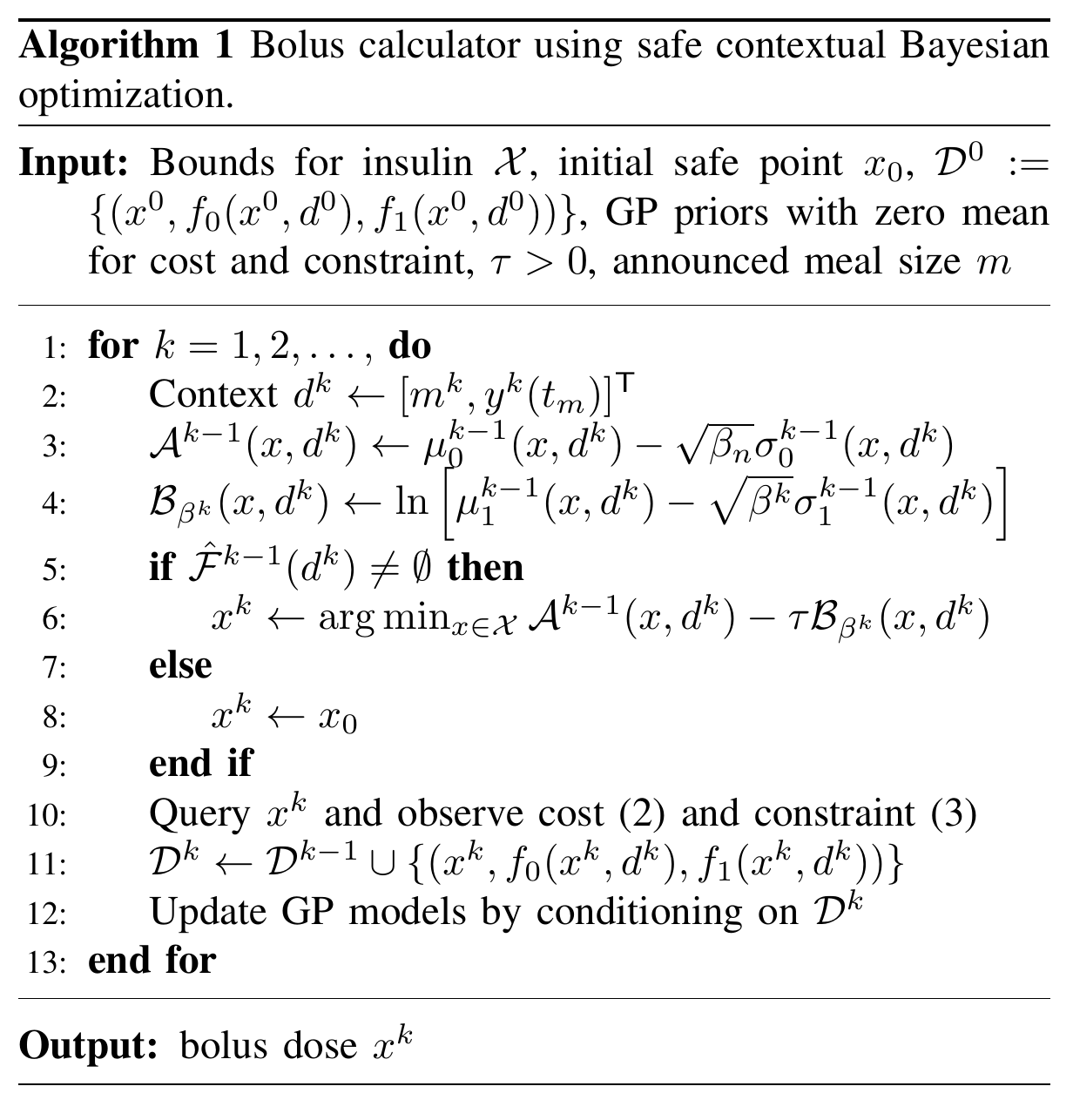}
\end{figure}

\section*{Appendix B - Individual patient data}

The blood glucose concentration of the individual patient data for the 10-adult cohort is shown in Fig.~2. 

\section*{Appendix C- Glossary of medical terms}
\begin{itemize}
	\item Postprandial glucose -   glucose concentrations after eating a meal.
	\item CGM - Continuous glucose monitor sensor
	\item mild hypoglycemia - Low blood glucose concentrations (54$<$BG$<$ 70 mg/dL)
	\item severe hypoglycemia - Critically low  blood glucose concentrations ($<$54 mg/dL)
	\item hyperglycemia - Elevated blood glucose concentration ($>$180mg/dL)
	\item diurnal pattern - Patterns occurring every 24 hours
	\item Bolus insulin - Fast acting insulin use to counteract the effect of meal
	\item Basal insulin - Insulin used to regulate the resting blood glucose concentration 
\end{itemize}
\begin{figure*}[h]
	\centering
		\centering
		\includegraphics[width=0.6\linewidth]{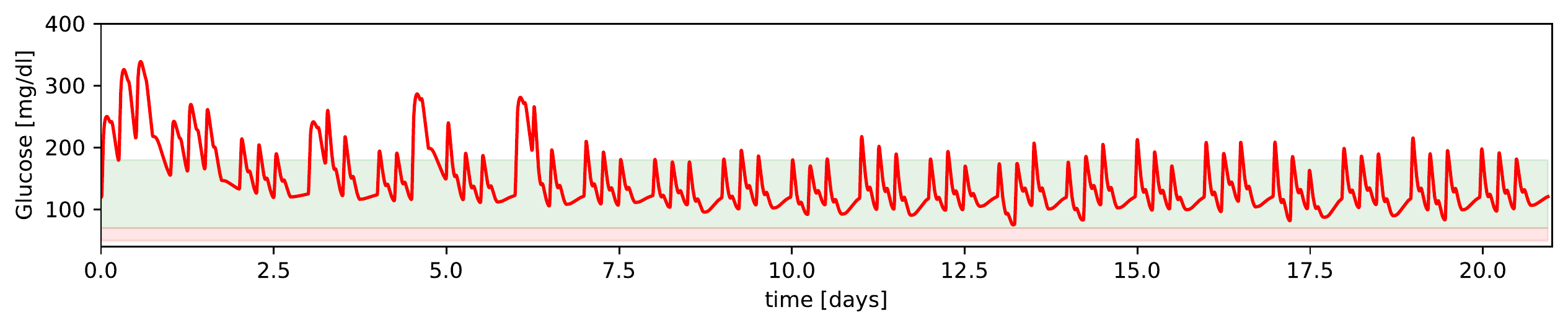}
		\includegraphics[width=0.6\linewidth]{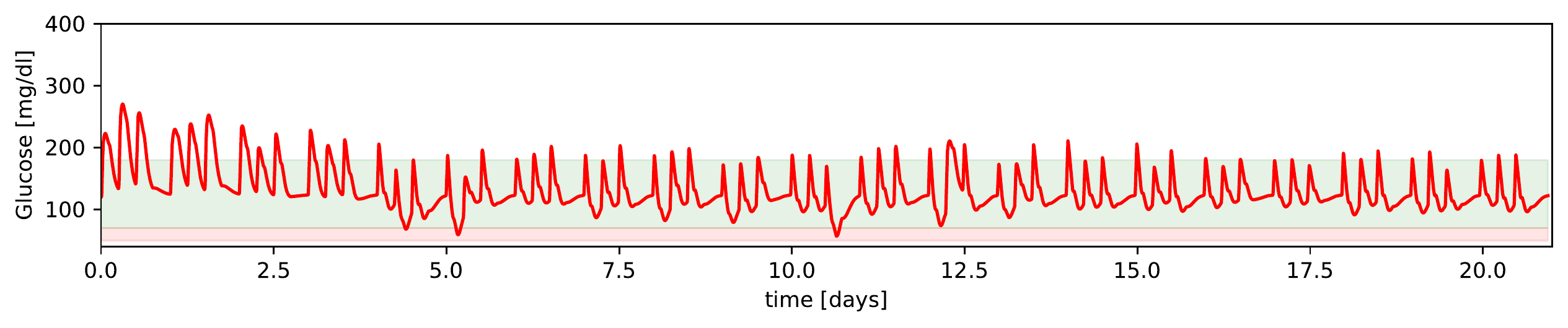}
		\includegraphics[width=0.6\linewidth]{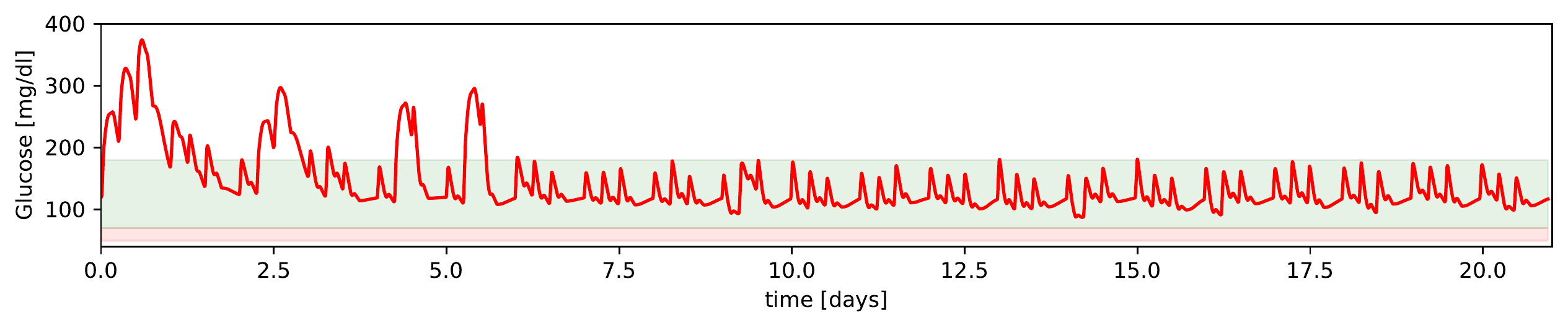}
		\includegraphics[width=0.6\linewidth]{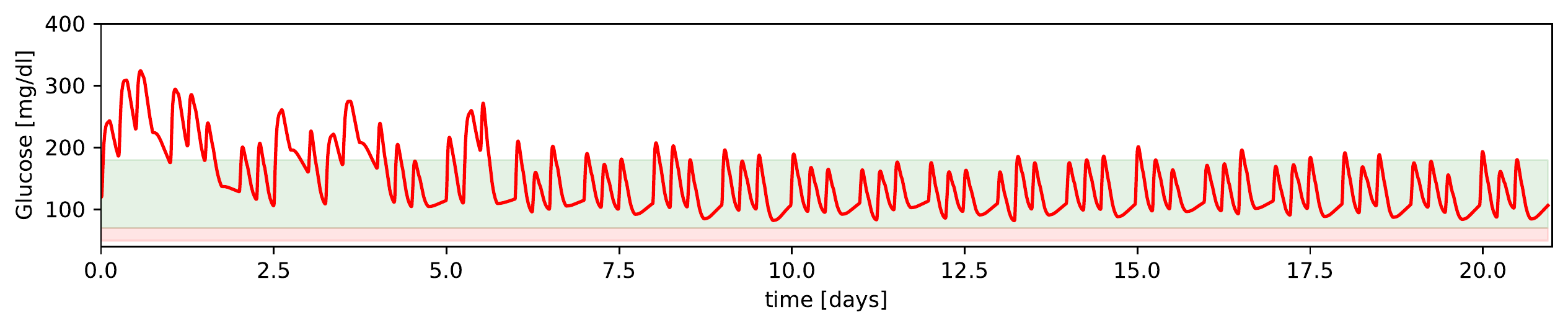}
		\includegraphics[width=0.6\linewidth]{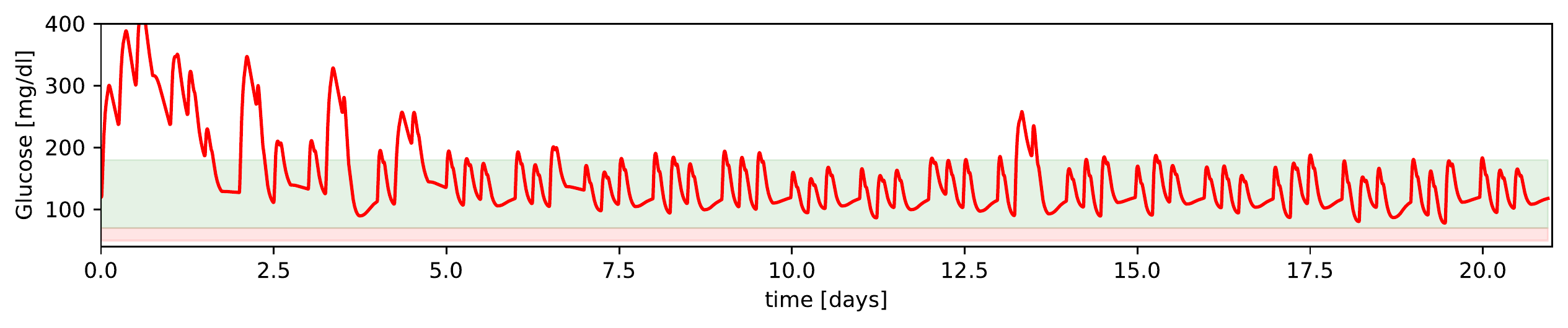}
		\includegraphics[width=0.6\linewidth]{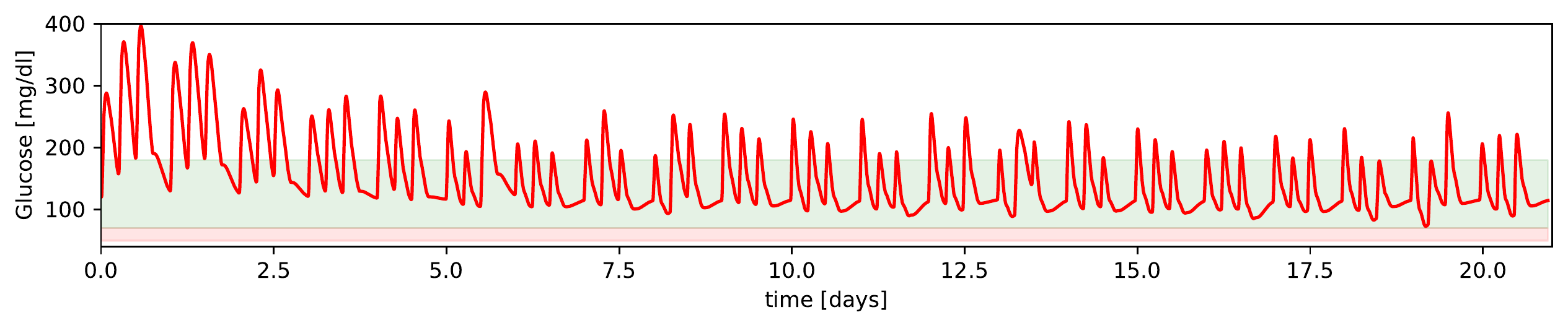}
		\includegraphics[width=0.6\linewidth]{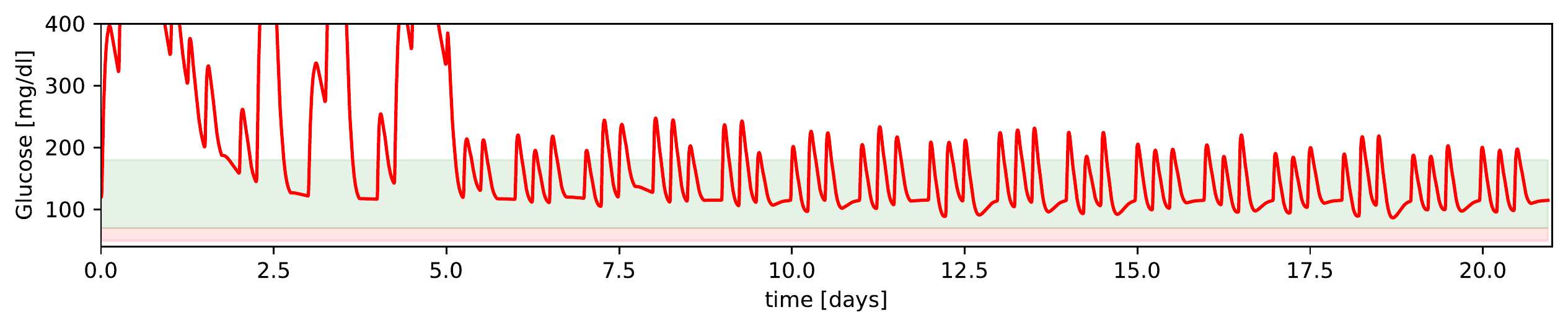}
		\includegraphics[width=0.6\linewidth]{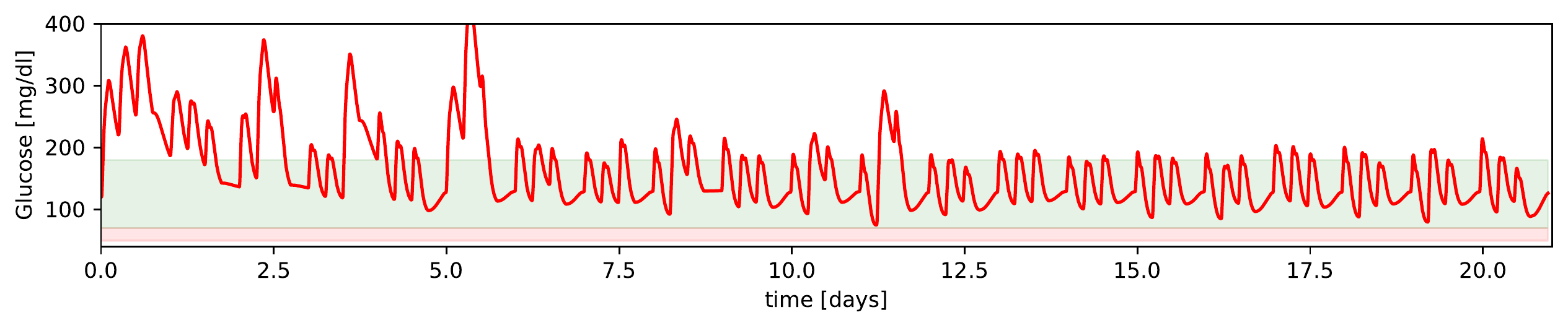}
		\includegraphics[width=0.6\linewidth]{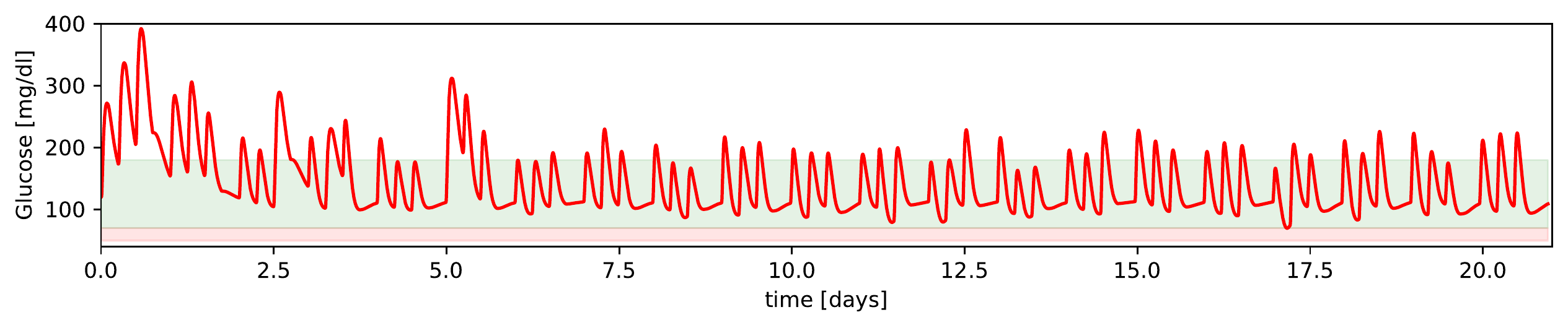}
		\includegraphics[width=0.6\linewidth]{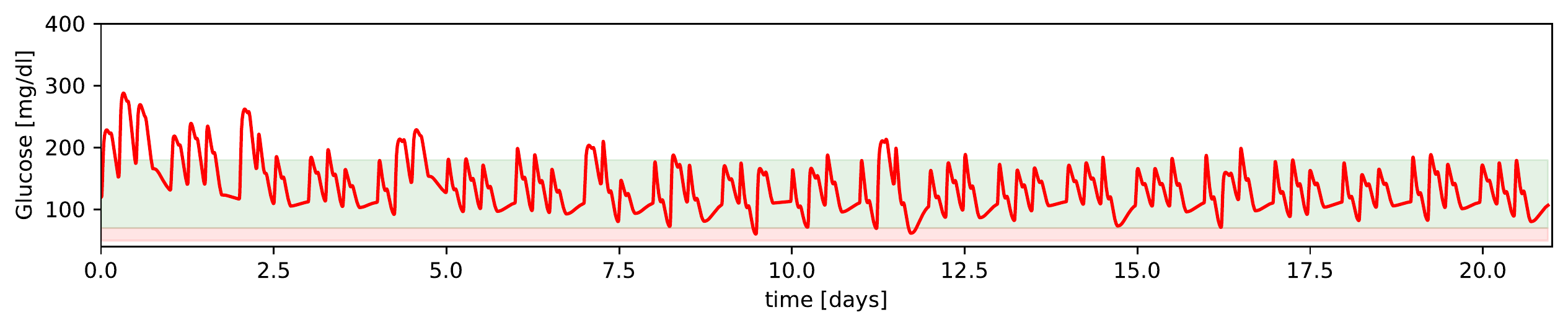}
	\caption{Blood glucose concentration over the three weeks with three meals each day.}\label{Fig:UVA protocolB BG}
\end{figure*}